\documentclass[12pt,a4paper]{amsart}

\usepackage{amssymb}
\usepackage{graphicx}
\usepackage{tikz}
\usetikzlibrary{calc,arrows,positioning}
\usetikzlibrary{matrix,decorations.pathreplacing,backgrounds,positioning}
\usetikzlibrary{calc}
\usepackage{amsmath}

\usepackage{float}
\usepackage{amstext,amsfonts,amsmath,amssymb}
\usepackage{amsthm}
\usepackage{graphicx}
\makeatletter
\def\@author#1{\g@addto@macro\elsauthors{\normalsize%
\def\baselinestretch{1}%
\upshape\authorsep#1\unskip\textsuperscript{%
\ifx\@fnmark\@empty\else\unskip\sep\@fnmark\let\sep=,\fi
\ifx\@corref\@empty\else\unskip\sep\@corref\let\sep=,\fi}%
\def\authorsep{\unskip,\space}%
\global\let\@fnmark\@empty
\global\let\@corref\@empty  
\global\let\sep\@empty}%
\@eadauthor={#1}
}
\makeatother

\newtheorem{theorem}{Theorem}[section]

\newtheorem{lemma}{Lemma}[section]
\newtheorem{example}{Example}[section]
\newtheorem{conjecture}{Conjecture}[section]
\newtheorem{remark}{Remark}[section]
\newtheorem{corollary}{Corollary}[section]

\def\dj{d\kern-0.4em\char"16\kern-0.1em}
\newcommand{\la}{\lambda}
\newcommand{\om}{\omega}
\newcommand{\x}{{\bf x}}

\newcommand{\y}{{\bf y}}
\newcommand{\mul}{{\rm mult}}
\newcommand{\nsg}{{\rm NSG}}
\newcommand{\dng}{{\rm DNG}}
\newcommand{\Sp}{{\rm Spec}}

\title[Vertex types in threshold and chain graphs]{\bf Vertex types in threshold and chain graphs}

\author{M. An\dj eli\'c}
\address{Department of Mathematics, Kuwait University, Safat 13060, Kuwait}
\email{milica@sci.kuniv.edu.kw}
\thanks{}

\author{E. Ghorbani}
\address{Department of Mathematics, K. N. Toosi University of Technology,  P. O. Box 16315-1618, Tehran, Iran, and
 School of Mathematics, Institute for Research in Fundamental
Sciences (IPM),  P. O. Box
19395-5746, Tehran, Iran}
\email{e\_ghorbani@ipm.ir}

\author{S.K. Simi\'c}
\address{State University of Novi Pazar, Vuka Karad\v{z}i\'{c}a bb, 36 300 Novi Pazar,
Serbia, and Mathematical Institute SANU, Kneza Mihaila 36,
11 000 Belgrade, Serbia} \email{sksimic@mi.sanu.ac.rs}
\thanks{}


\dedicatory{In honour of Domingos M. Cardoso on the occasion of his
65th birthday}


\begin{document}

\begin{abstract}
A graph is called a chain graph if it is bipartite and the
neighborhoods of the vertices in each color class form a chain with
respect to inclusion. A threshold graph can be obtained from a chain
graph by making adjacent all pairs of vertices in one color class.
Given a  graph $G$, let $\la$ be an eigenvalue (of the adjacency
matrix) of $G$ with multiplicity $k \geq 1$. A vertex $v$ of $G$ is
a downer, or neutral, or  Parter depending whether the multiplicity
of $\la$ in $G-v$
 is $k-1$, or $k$, or $k+1$, respectively. We consider vertex types in the above sense in threshold and chain graphs. In particular,
 we show that chain graphs can have neutral vertices, disproving a conjecture by Alazemi {\em et
 al.}

\end{abstract}

\maketitle

\section{Introduction}

\noindent This paper is a successor of \cite{AFGS} in which vertex
types (see the Abstract) in the lexicographic products of an
arbitrary graph over cliques and/or co-cliques were investigated.
Such class of graphs includes threshold graphs and chain graphs as
particular instances.
 Both of
these types (or classes) of graphs were discovered, and also
rediscovered by various researchers in different contexts (see, for
example, \cite{bcrs, bfp,hpx}, and references therein). Needles to
say, they were named by different names mostly depending on
applications in which they arise. It is also noteworthy that
threshold  graphs are subclass of cographs, i.e. of $P_4$-free
graphs. Recall that threshold graphs are $\{P_4,2K_2,C_4\}$-free
graphs, while chain graphs are $\{2K_2,C_3,C_5\}$-free graphs - see
\cite{AlAS, Pt-srb} for more details.   Note, if these graphs are
not connected then (since $2K_2$ is forbidden) at most one of its
components is non-trivial (others are trivial, i.e. isolated
vertices). Moreover, stars are the only connected graphs which
belong to both  of two classes of graphs under consideration.

Recall, these graphs play a very important role in Spectral Graph
Theory, since the maximizers for the largest eigenvalue of the
adjacency matrix (for graphs of fixed order and size, either
connected or disconnected) belong to these classes (threshold graphs
in general case, and chain graphs in bipartite case). Such graphs
(in both classes) have a very specific structure (embodied in
nesting property), and this fact enables us to tell more on the type
of certain vertices. Here, we also disprove Conjecture 3.1 from
\cite{Pt-srb}.

Throughout, we will consider simple graphs, i.e. finite undirected
graphs without loops or multiple edges.  In addition, without loss
of generality, we will assume that any such graph is connected. For
a graph $G$ we denote its vertex set by $V(G)$, and  by $n=|V(G)|$
its {\em order}.
An $n \times
n$ matrix $A(G) = [a_{ij}]$ is its {\em adjacency matrix} if $a_{ij}
= 1$ whenever vertices $i$ and $j$ are adjacent, or $a_{ij} = 0$
otherwise. For a vertex $v$ of $G$, let $N(v)$ denote the {\em neighborhood} of $v$, i.e.   the set of
all vertices of $G$ adjacent to $v$.

The eigenvalues of $G$ are the eigenvalues of its adjacency matrix.
In non-increasing order they are denoted by
\[
\lambda_1(G)\geq \lambda_2(G) \geq \cdots \geq \lambda_n(G),
\]
or by
\[
 \mu_1(G) > \mu_2(G)> \cdots > \mu_r(G)
 \]
if only distinct eigenvalues are considered.  If understandable from
the context we will drop out graph names from the notation of
eigenvalues (or other related objects). The eigenvalues comprise
(together with multiplicities, say $k_1,k_2,\ldots,k_r$,
respectively) the spectrum of $G$, denoted by $\Sp(G)$. The {\em
characteristic polynomial} of $G$, denoted by $\phi(x;G)$, is the
characteristic polynomial of its adjacency matrix. Both, the
spectrum and characteristic polynomial of a graph $G$ are its
invariants. Further on, all spectral invariants (and other relevant
quantities) associated to the adjacency matrix will be prescribed to
the corresponding graph. For a given eigenvalue $\la \in \Sp(G)$,
$\mul(\la,G)$ denotes its multiplicity, while ${\mathcal E}(\la;G)$
its eigenspace (provided $G$ is a labeled graph). The equation $A
{\mathbf x} = \la {\mathbf x}$, is called the eigenvalue equation
for $\la$. Here $A$ is the adjacency matrix, while  ${\mathbf x}$ a
$\la$-eigenvector also of the labeled graph $G$. If $G$ is of order
$n$, then ${\mathbf x}$ can be seen as an element of ${\mathbb
R}^n$, or a mapping ${\mathbf x} : V(G) \rightarrow {\mathbb R}^n$
(so its $i$-th entry can be denoted by $x_i$ or ${\mathbf x}(i)$).
Eigenspaces (as the eigenvector sets) are not graph invariants,
since the eigenvector entries become permuted if the vertices of $G$
are relabeled.

An eigenvalue $\la \in \Sp(G)$ is {\em main}  if the corresponding eigenspace ${\mathcal E}(\la;G)$ is not orthogonal to
all-$1$ vector $\mathbf{j}$; otherwise, it is {\em non-main}.

Given a graph $G$, let $\la$ be its eivgenvalue of multiplicity $k
\geq 1$ and $v\in V(G)$. Then $v$ is a {\em downer}, or {\em
neutral}, or {\em Parter} vertex of $G$, depending whether the
multiplicity of $\la$ in $G-v$ is $k-1$, or $k$, or $k+1$,
respectively. Recall, neutral and Parter vertices of $G$ are also
called {\em Fiedler} vertices. For more details, about the above
vertex types see, for example, \cite{SAFZ}.

\begin{remark} Sum rule: Let $\x$ be a $\la$-eigenvector of a graph $G$. Then the entries of $\x$ satisfy the following equalities:
\begin{equation}\label{sumrule}
\la\x(v)=\sum_{u\sim v}\x(u),~~\hbox{for all}~v\in V(G).
\end{equation}
From \eqref{sumrule} it follows that if $\la\ne0$, then $N(u)=N(v)$
implies that $\x(u)=\x(v)$ and if $\la\ne-1$,
$N(u)\cup\{u\}=N(v)\cup\{v\}$ implies that $\x(u)=\x(v)$. 
\end{remark}

In sequel, we will need the following interlacing property for graph eigenvalues (or, eigenvalues of Hermitian matrices, see \cite[Theorem~2.5.1]{bh}).

\begin{theorem}\label{inter}
Let $G$ be a graph of order $n$ and $G'$ be an induced subgraph of $G$ of order $n'$. If
$\la_1 \ge \la_2\ge \cdots \ge \la_n$ and
$\la'_1 \ge \la'_{2} \ge \cdots \ge \la'_{n'}$ are their eigenvalues respectively, then
\begin{equation}{\label{jed}}
\la_i \ge \la'_i \ge \la_{n-n'+i}~~~\hbox{for}~ i=1,2, \ldots,n'.
\end{equation}
In particular, if $n'=n-1$, then $$\la_1\ge\la'_1\ge\la_2\ge\la'_2\ge\cdots\ge\la_{n-1}\ge\la'_{n-1}\ge\la_n.$$
\end{theorem}
In the case of equality in (\ref{jed}) (see
\cite[Theorem~2.5.1]{bh}) the following holds.

\begin{lemma}\label{eqinter} If $\la'_i=\la_i$ or $\la'_i=\la_{n-n'+i}$
for some $i \in \{1,2,\ldots,n'\}$, then $G'$ has an eigenvector $\x'$ for $\la'_i$ such that
$\begin{pmatrix}\bf0 \\ \x'\end{pmatrix}$ is an eigenvector of $G$ for $\la'_i$, where
 $\bf0$ is a zero vector whose entries correspond to the vertices from  $V(G)\setminus V(G')$.
\end{lemma}

\begin{remark}\label{downer} A vertex $v$ is a downer
for a fixed eigenvalue $\lambda$, if there exists in the
corresponding eigenspace an eigenvector whose $v$-th component is
non-zero. Otherwise, it is a Fiedler vertex. Let $W$ be the
eigenspace corresponding to $\la$. If for each $\x\in W$, we have
$\x(v)=0$, then $v$ cannot be a downer vertex as for any $\x\in W$,
the vector $\x'$ obtained by deleting the $v$-th component, is a
$\la$-eigenvector of $G-v$, and therefore we have
  $$\mul(\la,G-v)\ge\dim\,\{\x':\x\in W\}=\dim W=\mul(\la,G).$$
From this and Lemma~\ref{eqinter} it follows, if $\mul(\la, G)=1$
that there exists a $\la$-eigenvector $\x$ with $\x(v)=0$ if and
only if $v$ is not a downer vertex for $\la$.
\end{remark}

\bigskip
The rest of the paper is organized as follows: in Section  \ref{th}
we give some particular results about vertex types in threshold
graphs, while in Section~\ref{cg} we put focus on chain graphs, and
among others we disprove  Conjecture 3.1 from \cite{AlAS}, which
states that in any chain graph, every vertex is a downer with
respect to every non-zero eigenvalue. Besides we point out that some
weak versions of the same conjecture are true.

\section{Vertex types in threshold graphs}\label{th}

\noindent Any (connected) threshold graph $G$ is a split graph i.e.,
it admits a partition of its vertex set into two subsets, say $U$
and $V$, such that the vertices of $ U$ induce a co-clique, while
the vertices of $ V$  induce a clique. All other edges join a vertex
in $U$ with a vertex in $V$. Moreover, if $G$ is connected, then
both $U$ and $V$ are partitioned into $h\geq 1$ non-empty cells such
that $U=\bigcup_{i=1}^h U_i$ and $V=\bigcup_{i=1}^h V_i$ and  the
following holds for (cross) edges: each vertex in $U_i$ is adjacent
to all vertices in $V_1\cup\cdots\cup V_i$ (a nesting property).
Accordingly, connected threshold graphs are also called {\em nested
split graphs} (or NSG for short). If $m_i = |U_i|$ and $n_i =|V_i|$,
then we write
\begin{equation}{\label{nsg}}
G =\nsg(m_1,\ldots , m_h; n_1, \ldots, n_h),
\end{equation}
(see Fig. \ref{thrg}). We denote by $M_h$ ($=\sum_{i=1}^h m_i$) the
size of $U$, and by $N_h$ ($=\sum_{i=1}^h n_i$) the size of $V$.


\begin{figure}[hbtp]
\centering

\scalebox{1.0}{
\begin{tikzpicture}[line width=0.8pt]
\tikzset{every node/.style={draw,shape=circle,minimum height=0.4cm,minimum width=0.4cm,inner sep=0pt,fill=none}}

\foreach \x/\i in {1/1,2.2/2,5/3,6.2/4} {\node[anchor=center] at (1,\x) (p\i) {};}
\node [draw=none,anchor=east] at (1,0.62) {\small$U_h$};
\node [draw=none,anchor=east] at (1,1.83) {\small$U_{h-1}$};
\node [draw=none,anchor=east] at (1,4.65) {\small$U_{2}$};
\node [draw=none,anchor=east] at (0.95,5.9) {\small$U_1$};
\node [draw=none,anchor=center] at (0.7,1.32) {\small$m_{h}$};
\node [draw=none,anchor=center] at (0.7,2.52) {\small$m_{h-1}$};
\node [draw=none,anchor=center] at (0.8,5.35) {\small$m_{2}$};
\node [draw=none,anchor=center] at (0.65,6.5) {\small$m_{1}$};
\foreach \x/\i in {1/1,2.2/2,5/3,6.2/4} {\node[anchor=center,fill] at (4.2,\x) (q\i) {};}
\node [draw=none,anchor=west] at (4.25,0.62) {\small$V_h$};
\node [draw=none,anchor=west] at (4.25,1.83) {\small$V_{h-1}$};
\node [draw=none,anchor=west] at (4.25,4.6) {\small$V_{2}$};
\node [draw=none,anchor=west] at (4.25,5.8) {\small$V_1$};
\node [draw=none,anchor=west] at (4.3,1.3) {\small$n_{h}$};
\node [draw=none,anchor=west] at (4.3,2.5) {\small$n_{h-1}$};
\node [draw=none,anchor=west] at (4.3,5.3) {\small$n_{2}$};
\node [draw=none,anchor=west] at (4.3,6.5) {\small$n_{1}$};
\foreach \i in {1,2,4} {\draw (p1) -- (q\i);}
\draw (p1.40) -- (q3.245);
\draw (p2) -- (q2);
\draw (p2.25) -- (q3);
\draw (p2) -- (q4.215);
\foreach \i in {3,4} {\draw (p3) -- (q\i);}
\draw (p4) -- (q4);
\draw (q3) -- (q4);
\draw (q1) -- (q2);
\draw (q1) to[out=0,in=0,looseness=0.75] (q4);
\draw (q2) to[out=0,in=340,looseness=0.75] (q4);
\draw (q1) to[out=160,in=190,looseness=0.6] (q3);
\draw (q2) to[out=160,in=205,looseness=0.6] (q3);
\draw (q2) -- ($(q2) + (0,0.8)$);
\draw (q3) -- ($(q3) + (0,-0.8)$);
%
\fill (1,3.9) circle (1.2pt) ++(0,-0.3) circle (1.2pt) ++(0,-0.3) circle (1.2pt);
\fill (4.2,3.9) circle (1.2pt) ++(0,-0.3) circle (1.2pt) ++(0,-0.3) circle (1.2pt);

\end{tikzpicture}
}

\caption{The threshold graph $G = \nsg(m_1,\ldots , m_h; n_1,
\ldots, n_h)$.} \label{thrg}
\end{figure}
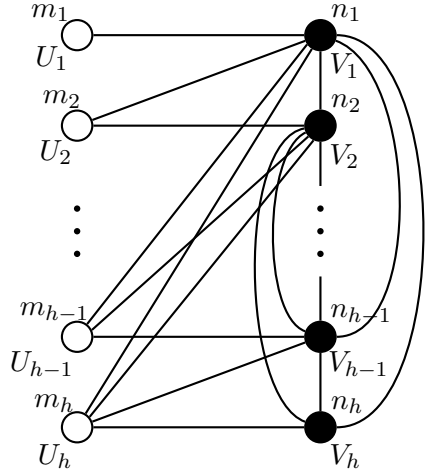

The following Theorem states the essential spectral properties of
threshold graphs (see \cite{AlAS, trev2, SciFar}).

\begin{theorem}\label{nsg_mult}
Let $G=\nsg(m_1, \ldots, m_h; n_1, \ldots, n_h)$. Then the spectrum
of $G$ contains:
\begin{itemize}
\item  $h$ positive simple eigenvalues;
\item $h-1$ simple eigenvalues less than $-1$ if $m_h=1$, or
otherwise if $m_h\geq 2$, $h$ simple eigenvalues less than $-1$;
\item eigenvalue $0$ of multiplicity $M_h-h$, and $-1$ of multiplicity $N_h-h+1$ if $m_h=1$, or of multiplicity
$N_h-h$ if $m_h > 1$.
\end{itemize}
In addition, if $\lambda \neq 0,-1$ then $\lambda$ is a main eigenvalue.
\end{theorem}

\begin{remark}
If $\lambda \neq 0,-1$, then any vertex of a threshold graph is
either  downer or  neutral. Parter vertices may arise only for
$\lambda = 0$ or $-1$.  Any vertex deleted subgraph $G-v$ of a
threshold graph $G$ is a threshold graph as well.  By Theorem
\ref{nsg_mult} one can easily determine the multiplicities of $0$
and $-1$ in both  $G$ and $G-v$ and, consequently, the vertex type
for $v$ of $\la=0$ or $-1$.
\end{remark}

\bigskip Recall that any vertex of a connected graph is downer
for the largest eigenvalue, see \cite[Proposition 1.3.9.]{crs}. In
addition, if $\la\neq 0,-1$, then the corresponding eigenvector $\x$
is unique (up to scalar multiple) and constant on each of the sets
$U_i$ and $V_i$ ($i=1,\ldots, h$); in particular, if $m_h=1$ then it
is constant on the set $U_h\cup V_h$. These facts will be used
repeatedly further on without any recall.

\begin{theorem}\label{downers}
Let $G=\nsg(m_1, \ldots, m_h; n_1, \ldots, n_h)$ and $\lambda \neq
0,-1$ its eigenvalue other than the largest one. Then all vertices
in $U_1 \cup V_1$ are downers for $\lambda$. The same holds for
vertices in $U_h$, and also in $V_h$ unless $\lambda = -m_h$ and
$m_h\geq 2$.
\end{theorem}

\begin{proof}
Let $u_1 \in U_1$ and $v_1 \in V_1$. Then, by the sum rule, $\lambda
\x(u_1) = n_1 \x(v_1)$. Since $\lambda \neq 0,-1$, $u_1$ and $v_1$
are both downer or Fiedler vertices (see Remark \ref{downer}). Let
$X = \sum_{w \in V(G)} \x(w)$, and by the way of contradiction
assume that  $u_1$ and $v_1$ are both Fiedler vertices, i.e.
$\x(u_1) = \x(v_1) = 0$. Again, by the sum rule, we have $\lambda
\x(v_1) =X - \x(v_1)$, and therefore $X = 0$, a contradiction since
$\lambda\neq 0,-1$ is a simple and non-main eigenvalue (see Theorem
\ref{nsg_mult}).

Let $u_h\in U_h$,  $v_h \in V_h$ and $Y = \sum_{w \in V_1 \cup
\cdots \cup V_h} \x(w)$. Then, $\lambda \x(u_h) = Y$ and $\lambda
\x(v_h) = Y - \x(v_h) + m_h \x(u_h)$. For a contradiction, let
$\x(u_h) = 0$. Then it easily follows that $\x(v_h) = 0$. We next
claim that for $2\le i \le h$, $\x(u_i) =\x(v_i) = 0$ implies
$\x(u_{i-1}) =\x(v_{i-1}) = 0$. To see this, since $\x(v_{i})=0$ by
the sum rule we obtain
 $$\lambda \x(u_{i})= Y-\sum_{j=i+1}^h n_j\x(v_j)= Y-\sum_{j=i}^h n_j\x(v_j)=\lambda\x(u_{i-1}),$$
  and therefore $\lambda\x(u_{i-1}) = 0$. Similarly, since
  $\x(u_i)=0$ and
 \begin{eqnarray*}
\lambda\x(v_{i})&=&Y-\x(v_{i})+\sum_{j=i}^hm_j\x(u_{j})\\
&=&Y-\x(v_{i})+\sum_{j=i-1}^hm_j\x(u_{j})=(\la+1)\x(v_{i-1}),
\end{eqnarray*}
 it follows $\x(v_{i-1})=0$. Consequently, we obtain $\x
(u_h)=\cdots=\x (u_1)=0$ and $\x (v_h)=\cdots=\x (v_1)=0$, i.e. $\x
=\mathbf{0}$, a contradiction. This proves that all vertices in
$U_h$ are downers for $\lambda$.

For the last part of the theorem, let $\lambda \neq - m_h$. Then we
have
$$\la\x(u_h)=Y,~~\la\x(v_h)=Y-\x(v_h)+m_h\x(u_h),$$
and so $(\la+1)\x(v_h)=(\la+m_h)\x(u_h)$.
Hence, if $\x(v_h) = 0$, then $\x(u_h) =0$ and we reach a contradiction as above.
 Consequently, all vertices in $V_h$ are downers.
\end{proof}

\begin{remark}{\label{mh}} The following example shows
that in unresolved case when $\lambda=-m_h$ and $m_h\geq 2$ vertices
in $V_h$ may be neutral.

Let $G=\nsg(2,2,2;2,3,2)$. Then all vertices in $U_3$ are downers,
while all vertices in  $V_3$ are neutral for $\lambda =-2$. So, an
unresolved case from Theorem \ref{downers} can be an exceptional
one.

So, the following question arises: Can we find an example when
$\la=-m_h$, $m_h\geq 2$ and that each vertex in $V_h$ is a downer?
\end{remark}

\begin{theorem}
Let $G=\nsg(m_1, \ldots, m_h; n_1, \ldots, n_h)$ and let $\lambda
\neq 0,-1$ be its eigenvalue. Then, for any $i=1,\ldots, h-1$, at
least one of $U_i$, $U_{i+1}$ (resp. $V_i$, $V_{i+1}$) contains only
downer vertices for $\lambda$.
\end{theorem}

\begin{proof}
Recall first that all vertices within $U_k$ or $V_k$ $(k=1,\ldots,
h)$ are of the same type for $\la$, and that $\la$ is  a simple
eigenvalue. Assume on the contrary that all vertices in $U_i$ and
$U_{i+1}$ are neutral and let $\x$ be a $\la$-eigenvector. Then, for
$u_i \in U_i$ and $u_{i+1}\in U_{i+1}$, $\x(u_i)=\x(u_{i+1})=0$. By
the sum rule it easily follows that for any $v_{i+1}\in V_{i+1},$
$\x(v_{i+1})=0$. Next, we have
\begin{eqnarray}
\lambda \x(v_{i})=\sum_{j=1}^h
n_j\x(v_j)-\x(v_{i})+\sum_{j=i}^h
m_j\x(u_j)\label{eq_i},\\
\lambda \x(v_{i+1})=\sum_{j=1}^h
n_j\x(v_j)-\x(v_{i+1})+\sum_{j=i+1}^h
m_j\x(u_j)\label{eq_i+1}.
\end{eqnarray}
By subtracting (\ref{eq_i+1}) from (\ref{eq_i}) we obtain $\lambda
\x(v_{i})=-\x(v_{i})$. Since $\lambda \neq -1$, $\x(v_{i})=0$ and
consequently $\x(u_{i-1})=0$. Proceeding in the similar way, we
conclude that $\x(u_1)=0$,  which contradicts Theorem \ref{downers}.

The proof for vertices in $V_i, V_{i+1}$ is similar, and therefore
omitted.
\end{proof}

Next examples show that in an nested split graph $G$ neutral
vertices for the same eigenvalue may be  distributed in different
$U_i$'s, $V_i$'s and at the same time in both $U$ and $V$.

\begin{example}
If $G=\nsg(4,1,3,1,1; 1,1,1,2,1)$, then all vertices in $U_2$ and
$U_4$ are neutral vertices for $\lambda_3=1$.

If $G=\nsg(2,4,4,2; 1,1,1,2)$, then all vertices in $V_2$ and $V_4$
are neutral for $\lambda_{16}=-2$.
\end{example}

\begin{example}
In $G=\nsg(2,2,5,1;1,1,1,1)$ all vertices in $U_3$ and in $V_2$ are
neutral vertices for  $\lambda_2=1$.
\end{example}

In what follows we assume that all vertices in $U_{s}$ (resp.
$V_{s}$) of a nested split  graph $G$ are neutral for some $s$ with
respect to some $\lambda_i \neq 0,-1$. If so, we will show that this
assumption imply some restrictions on position of $\lambda_i$ in the
spectrum of $G$.

\begin{theorem}\label{strict-inter}
Let $G=\nsg(m_1, \ldots, m_h; n_1, \ldots, n_h)$ such that  all
vertices in $U_{s}$ for some $2\leq s\leq h-1$ are neutral for
$\lambda_i\neq 0,-1$. If $G'=\nsg(m_{s+1}, \ldots, m_h; n_{s+1},
\ldots, n_h)$, $n'=|V(G')|$ and $\Sp(G')=\{\la'_1, \ldots,
\la'_{n'}\}$ then $\lambda_i=\la'_j$ for some $j\in \{1,
\ldots,n'\}$. Moreover, $j<i<n-n'+j$, and if $i\leq n'$ then
$\lambda_i\neq \la'_{n'}$.
\end{theorem}

\begin{proof}
Let $G'$ be the induced subgraph of $G$ obtained by deleting all
vertices in $U_1, \ldots, U_s, V_1, \ldots, V_s$ i.e.
$$G'=\nsg(m_{s+1},\ldots, m_h; n_{s+1},\ldots, n_h).$$
Let $n'=|V(G')|=\sum_{j=s+1}^h(m_j+n_j)$ and let $\x$ be a
$\la$-eigenvector of $G$. Denote by  $\mathbf{x'}$ the vector
obtained from $\x$ by deleting all entries corresponding to deleted
vertices from $G$. Since
$$0=\lambda_i \x(u_s)=\sum_{j=1}^s n_j\x(v_j),$$ for
any $k\geq s+1$, we obtain
\begin{eqnarray*}
\lambda_i\mathbf{x'}(u_k)&=&\lambda_i\x(u_k)=\sum_{j=1}^kn_j\x(v_j)=\sum_{j=s+1}^k n_j\x(v_j)=\sum_{j=s+1}^k n_j\mathbf{x'}(v_j)\\
\lambda_i\mathbf{x'}(v_k)&=&\lambda_i\x(v_k)=\sum_{j=1}^hn_j\x(v_j)-\x(v_k)+\sum_{j=k}^hm_j\x(u_j)\\
&=&\sum_{j=s+1}^h n_j
\mathbf{x'}(v_j)-m_j\mathbf{x'}(v_k)+\sum_{j=k}^h \mathbf{x'}(u_j)
\end{eqnarray*}
and therefore $\mathbf{x'}$ is an eigenvector of $G'$ for
$\lambda_i$,  i.e. $\lambda_i\in\Sp(G')$. Suppose $\lambda_i=\la'_j$
for some $j\in \{1,\ldots, n'\}$. From interlacing it follows that
\begin{equation}\label{inter1}
\lambda_{n-n'+i}\leq \la'_i\leq \lambda_i=\la'_j, \quad
\mbox{if}\quad i\leq n'.
\end{equation}
as well as
\begin{equation}\label{inter1'}
\lambda_{n-n'+j}\leq \lambda_i=\la'_j\leq \lambda_j.
\end{equation}

If in \eqref{inter1},  at least one of inequalities holds as an
equality then, by Lemma \ref{eqinter}, $G'$ has an eigenvector
$\mathbf{y'}$ for $\la'_i$ such that $\left(
\begin{matrix}
\mathbf{0}\\
\mathbf{y'}
\end{matrix}
\right)$ is an eigenvector of $G$ for $\la'_i$. By the sum rule for
any vertex in $V_{s}$ we obtain that the sum of all entries of
$\mathbf{y'}$ is $0$ and accordingly that $\la'_i$ is non-main
eigenvalue of $G'$. Hence, $\la'_i=0$ or $\la'_i=-1$ which implies
$\la'_i<\la_i$. Similarly,  in \eqref{inter1'} we conclude that
$\la'_j=\la_i$ is a non-main eigenvalue of $G'$, a contradiction, by
Theorem \ref{nsg_mult}.
 Therefore, the interlacing in these
cases reads
\begin{eqnarray}
\lambda_{n-n'+i}\leq\la'_i< \lambda_i, \quad  i\leq n',\\
\lambda_{n-n'+j}< \la'_j=\la_i<\lambda_j\label{inter2}.
\end{eqnarray}
Moreover,  (\ref{inter2}) implies $j<i<n-n'+j$. Also, if $i\leq n'$,
$\lambda_i\neq \la'_{n'}$ holds. Otherwise,
$\lambda_{n-n'+i}\leq\la'_{i}<\lambda_i= \la'_{n'}$, a
contradiction.

\end{proof}

If all vertices  in $V_{s}$ for some $s$ are neutral for
$\lambda_i\neq 0,-1$, then bearing in mind that
$$G-V_s=\nsg(m_1, \ldots, m_{s-1}+m_{s}, \ldots, m_h; n_1,
\ldots, n_{s-1}, n_{s+1}, \ldots, n_h)$$ we can similarly conclude
the following.

\begin{theorem}\label{strict-inter2}
Let $G=\nsg(m_1, \ldots, m_h; n_1, \ldots, n_h)$ such that  all
vertices in $V_{s}$ for some $2\leq s\leq h$ are neutral for
$\lambda_i\neq 0,-1$. If
$$H_s=\nsg(m_1, \ldots, m_{s-1}+m_{s}, \ldots, m_h; n_1,
\ldots, n_{s-1}, n_{s+1}, \ldots, n_h),$$ and $\Sp(H_{s})=\{\la'_1,
\ldots, \la'_{n-n_{s}}\}$, then $\lambda_i=\la'_j$ for some $j\in
\{1, \ldots, n-n_{s}\}$. Moreover, $j<i<n_{s}+j$ and if $i\leq
n-n_s$ then $\lambda_i\neq \la'_{n-n_s}$.
  \end{theorem}

\begin{corollary}
Let $G=\nsg(m_1, \ldots, m_h; n_1, \ldots, n_h)$ of order $n$. Then
all vertices in $V(G)$ are downer vertices for
$\lambda_n$. 
\end{corollary}

\begin{proof}
If $\la_n=-1$, then $G$ is a complete graph and all vertices are
downers for it. So, we assume that $\la_n\ne0,-1$. Suppose on the
contrary that there exists at least one neutral vertex $u$ for
$\lambda_n$. If $u\in U_s$, then $\mathbf{x}(u)=0$,  where
$\mathbf{x}$ is a $\la$-eigenvector of $G$. As shown in the proof of
Theorem \ref{strict-inter}, $\lambda_n=\la'_j\in\Sp(G')$, for some
$j\in\{1, \ldots, n'\}$ and $\lambda_{n-n'+j}<\la'_j< \lambda_j$,
i.e. $\lambda_{n-n'+j}<\lambda_n< \lambda_j$, a contradiction.

The proof is similar if $v \in V_s$   for some $s$ and hence omitted
here.
\end{proof}

\begin{theorem}
Let $G=\nsg(m_1, \ldots, m_h; n_1, \ldots, n_h)$ such that all
vertices in $U_{s}$ are neutral for $\lambda\neq 0,-1$ and
$G''=\nsg(m_1, \ldots, m_{s}; n_1, \ldots, n_{s})$. Then
$$\lambda_{n''}(G'')<\lambda <\lambda_1(G''),$$ where $n''=|V(G'')|=\sum_{i=1}^s (m_i+n_i)$.
\end{theorem}

\begin{proof}
The graph $G''$ is an induced subgraph of $G$ with vertex set
$V(G'')=\bigcup_{j=1}^{s} (U_j\cup V_j)$. The adjacency matrix $A$
of the whole graph is equal to:
$$\left[\begin{matrix}
A'' & B\\
B^T& A'
\end{matrix}\right],$$
where $A', A''$ are adjacency matrices of $$G'=\nsg(m_{s+1},\ldots,
m_h; n_{s+1},\ldots, n_h)$$ and $G''$, respectively, and
$$B=\left[
\begin{matrix}
O_{M_s,n'}\\
J_{N_s, n'}
\end{matrix}\right],$$
where $M_s=\sum_{j=1}^s m_j$, $N_s=\sum_{j=1}^s n_j$ and
$n'=|V(G')|$. The corresponding eigenvector $\x$ can be represented
as $\x=\left(
\begin{matrix}
\x_1\\
\x_2 \end{matrix} \right)$ and the eigenvalue system reads:
\begin{eqnarray}
A''\x_1+B\x_2&=&\lambda \x_1\label{eq*}\\
B^T\x_1+A'\x_2&=&\lambda \x_2\label{eq**}.
\end{eqnarray}
As we have seen in the proof of Theorem \ref{strict-inter},
$\lambda$ is an eigenvalue of $A'$, the corresponding eigenvector is
$\x_2$ and further $\x_1\neq 0$. Therefore, it follows that
$B^T\x_1=0$, i.e. the sum of  some entries of $\x_1$ is $0$. From
(\ref{eq*}), we obtain
$$(\lambda I- A'')\x_1=B\x_2$$ and then by multiplying by $\x_1^T$ from
the left we obtain
$$\x_1^T(\lambda I- A'')\x_1=0$$
and consequently
$$\min_{\y\neq 0}\frac{\y^T(\lambda I- A'')\y}{\y^T\y}\leq
\frac{\x_1^T(\lambda I-A'')\x_1}{\x_1^T\x_1}\leq \max_{\y\neq 0}\frac{\y^T(\lambda I- A'')\y}{\y^T\y}.$$ Hence,
$$\lambda_{n''} (\lambda I- A'')\leq 0 \leq \lambda_1 (\lambda
I- A''),$$ where $n''=|V(G'')|=M_s+N_s$. Since, $\lambda_{n''}
(\lambda I- A'')=\lambda-\lambda_1(G'')$ and $\lambda_1 (\lambda I-
A'')=\lambda-\lambda_{n''}(G'')$ it follows
\begin{equation}\label{ineq}
\lambda_{n''}(G'')\leq \lambda\leq \lambda_1(G'').
\end{equation}
Moreover, $\lambda\neq \lambda_1(G'')$. Equality holds if and only
if $\x_1$ is an eigenvector of $G''$ for $\lambda_1(G'')$, that is
not possible due to the condition (\ref{eq**}) and positivity of
$\x_1$ as an eigenvector corresponding to the largest eigenvalue of
a connected graph. Similarly, if $\lambda=\lambda_{n''}(G'')$, then
$\x_1$ is the corresponding eigenvector and from (\ref{eq*}) it
follows $B\x_2=0$. This implies that $\lambda$ is a non-main
eigenvalue of a nested split  graph  $G'$,  a contradiction by
Theorem \ref{nsg_mult}.
\end{proof}

\begin{corollary}
Let \begin{eqnarray*}
G&=&\nsg(m_1, \ldots, m_h; n_1, \ldots, n_h),\\
G_s''&=&\nsg(m_1, \ldots, m_s; n_1, \ldots, n_s),
\end{eqnarray*}
$I_s=(\lambda_{n_s''}(G_s''),\lambda_1(G_s''))$, where
$n_{s''}=|(V(G_s'')|$ and $\lambda\in \Sp(G)$. If $\lambda\notin
\bigcup_{s=2}^{h-1} I_s$, then all vertices in $U$ are downer
vertices for $\lambda$.
\end{corollary}

\begin{example}
Let $G=\nsg(1,1,5; 1,1,8)$. Then $I_2=(-1.48,2.17)$ and besides
$\lambda_1$ and $\lambda_n$ all vertices in $U$ are downer for
$\lambda_{n-2}$ and $\lambda_{n-1}$, as well.
\end{example}

\section{Vertex types in chain graphs}\label{cg}

\noindent Chain graph can be defined as follows: a graph is a chain
graph if and only if it is bipartite and the neighborhoods of the
vertices in each color class form a chain with respect to inclusion.
For this reason, if connected (as was the case with threshold
graphs),  it is also called {\em double nested graph} \cite{bcrs}.

Non-zero eigenvalues of chain graphs are simple (see Theorem~\ref{dng_mult} below). As the subgraphs of any chain graph are also chain graphs,
it follows that there is no Parter vertex in any chain graphs with respect to non-zero eigenvalues.
A question raises whether they can have neutral vertices. In \cite{AlAS} it is conjectured that this cannot be the case.

\begin{conjecture}\label{conj} {\rm (\cite{AlAS})} In any chain graph, every vertex is downer with respect to every non-zero eigenvalue.
\end{conjecture}

We disprove Conjecture~\ref{conj} in this section. Indeed,
Theorems~\ref{1} and \ref{om} will show that there are infinitely
many counterexamples for this conjecture. In spite of that, a couple
of weak versions of the conjecture are true.


\begin{remark}\label{struc} {\em (Structure of chain graphs)} As it was observed in \cite{bcrs},
the color classes of any chain graph $G$ can be partitioned into $h$
non-empty cells $U_1,\ldots, U_h$ and $V_1,\ldots, V_h$  such that
$N(u)=V_1\cup\cdots\cup V_{h-i+1}~~\hbox{for any}~ u\in U_i,~1\le
i\le h.$ If $m_i=|U_i|$ and $n_i=|V_i|$, then we write
$\dng(m_1,\ldots, m_h;n_1, \ldots, n_h)$ (see Fig. \ref{chgr}).
\end{remark}

\begin{figure}[hbtp]
\centering

\scalebox{1.0}{
\begin{tikzpicture}[line width=0.8pt]
\tikzset{every node/.style={draw,shape=circle,minimum
height=0.4cm,minimum width=0.4cm,inner sep=0pt,fill=none}}

\foreach \x/\i in {1/1,2.2/2,5/3,6.2/4} {\node[anchor=center] at
(1,\x) (p\i) {};}
\node [draw=none,anchor=east] at (1,0.62) {\small$U_1$}; \node
[draw=none,anchor=east] at (1,1.83) {\small$U_2$}; \node
[draw=none,anchor=east] at (1.2,4.65) {\small$U_{h\!-\!1}$}; \node
[draw=none,anchor=east] at (0.95,5.9) {\small$U_h$};
\node [draw=none,anchor=center] at (0.7,1.32) {\small$m_{1}$}; \node
[draw=none,anchor=center] at (0.7,2.52) {\small$m_{2}$}; \node
[draw=none,anchor=center] at (0.8,5.35) {\small$m_{h\!-\!1}$}; \node
[draw=none,anchor=center] at (0.65,6.5) {\small$m_{h}$};
\foreach \x/\i in {1/1,2.2/2,5/3,6.2/4} {\node[anchor=center] at
(4.2,\x) (q\i) {};} \node [draw=none,anchor=west] at (4.25,0.62)
{\small$V_h$}; \node [draw=none,anchor=west] at (4.25,1.83)
{\small$V_{h-1}$}; \node [draw=none,anchor=west] at (4.25,4.6)
{\small$V_{2}$}; \node [draw=none,anchor=west] at (4.25,5.8)
{\small$V_1$};
\node [draw=none,anchor=west] at (4.3,1.3) {\small$n_{h}$}; \node
[draw=none,anchor=west] at (4.3,2.5) {\small$n_{h-1}$}; \node
[draw=none,anchor=west] at (4.3,5.3) {\small$n_{2}$}; \node
[draw=none,anchor=west] at (4.3,6.5) {\small$n_{1}$};
\foreach \i in {1,2,4} {\draw (p1) -- (q\i);} \draw (p1.40) --
(q3.245); \draw (p2) -- (q2); \draw (p2.25) -- (q3); \draw (p2) --
(q4.215); \foreach \i in {3,4} {\draw (p3) -- (q\i);} \draw (p4) --
(q4);
%
%
\fill (1,3.9) circle (1.2pt) ++(0,-0.3) circle (1.2pt) ++(0,-0.3)
circle (1.2pt); \fill (4.2,3.9) circle (1.2pt) ++(0,-0.3) circle
(1.2pt) ++(0,-0.3) circle (1.2pt);

\end{tikzpicture}
}

\caption{The chain graph $G=\dng(m_1, \ldots, m_h;n_1, \ldots, n_h)$.}%
\label{chgr}
\end{figure}
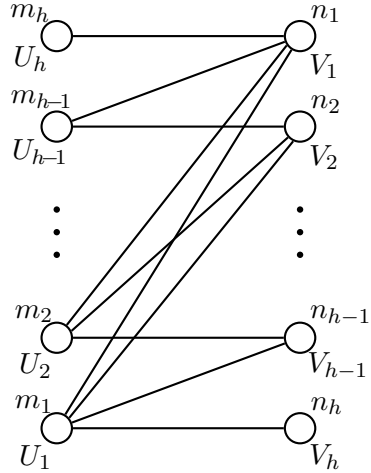


The spectrum of any chain graph has the following properties (see
\cite{AlAS}):
\begin{theorem}\label{dng_mult}
Let $G=\dng(m_1, \ldots, m_h; n_1, \ldots, n_h)$. Then the spectrum
of $G$ is symmetric about the origin and it contains:
\begin{itemize}
\item  $h$ positive simple eigenvalues greater then $\frac{1}{2}$;
\item $h$ negative simple eigenvalues less than $-\frac{1}{2}$;
\item eigenvalue $0$ of multiplicity $M_h+N_h-2h$.
\end{itemize}
\end{theorem}

\begin{remark}
On the contrary from threshold graphs  nonzero eigenvalues of chain
graphs need not be main. For more information see \cite{Pt-srb}.
\end{remark}

\begin{theorem}\label{chain} Let $G=\dng(m_1, \ldots, m_h; n_1, \ldots, n_h)$ be a chain graph. Then the vertices in $U_1\cup U_h\cup V_1\cup V_h$ are downer
 for any non-zero eigenvalue.
\end{theorem}
\begin{proof}Let $\x$ be any $\la$-eigenvector of $G$. Assume that $u_1\in U_1$ and $v_h\in V_h$. By the sum rule
$\lambda \x(v_h)=m_1\x(u_1)$. Since, $\la\neq 0$, $u_1$ and $v_h$
are both downer or neutral. Let $X=\sum_{w\in V}\x(w)$ and assume on
the contrary that $\x(u_1)=\x(v_h)=0$. Again, by the sum rule $\la
\x(u_2)=X-n_h\x(v_h)=0$  and consequently $\x(u_2)=0,$ for any
$u_2\in U_2$ as well as $\x(v_{h-1})=0$ for any $v_{h-1}\in
V_{h-1}$. Next, for any $u_3\in U_3$,
$$\la\x(u_3)=X-n_{h-1}\x(v_{h-1})-n_h\x(v_h)=0$$
It follows that $\x$ is zero on $U_3$, too. Continuing this
argument, it follows that $\x=\bf0$, a contradiction.\end{proof}

The following proposition states some facts related to vertex types
in chain graphs. The proofs are similar to those  in Section
\ref{th} and therefore omitted here.

\begin{theorem}
Let
\begin{eqnarray*}
G&=&\dng(m_1, \ldots, m_h;n_1, \ldots, n_h),\\
G_s'&=&\dng(m_1, \ldots,m_{s-1}; n_{h-s+2}, \ldots, n_h),\\
G_s''&=&\dng(m_s, \ldots, m_h; n_1, \ldots, n_{h-s+1}),
\end{eqnarray*}
$1<s<h$, $\la_i\in \Sp(G)\setminus\{0\}$, $n_s'=\sum_{j=1}^{s-1}
(m_j+n_{h-j+1})$, $n_s''=n-n_s'$, $\Sp(G_{s}')=\{\la'_1,\ldots,
\la'_{n_s'}\}$. Then
\begin{itemize}
\item
For any $j=1,\ldots, h-1$ at least one of $U_j$, $U_{j+1}$ contains
only downer vertices for $\lambda_i$.

\item If all vertices in $U_s$ for some $2<s<h-1$ are neutral
for $\lambda_i$, then \begin{itemize}
\item $\lambda_i$ is an eigenvalue
of $G_s'$ and $\la_i=\la'_j$,  for some $j\in \{1,\ldots, n_{s}'\}$.
If $\lambda_i$ is main, then $j<i< n-n_s'+j$. If $i\leq n_s'$ then
$\lambda_i\neq \la'_{n_s'}$.
\item
$\lambda_i\in[\lambda_{n_s''}(G_s''),\lambda_1(G_s''))$.
\item If $\la_i$ is a main eigenvalue then $\lambda_i\in (\lambda_{n_s''}(G_s''),\lambda_1(G_s''))$.
\item If $\lambda_i\notin \bigcup_{s=2}^{h-1}
[\lambda_{n_s''}(G_s''),\lambda_1(G_s''))$ then all vertices in
$V(G)$ are downer vertices for $\lambda_i$.
\end{itemize}
\end{itemize}
\end{theorem}

A chain graph for which $|U_1|=\cdots=|U_h|=|V_1|=\cdots=|V_h|=1$ is
called a {\em half graph}. Here we denote it by $H(h)$. As we will
see in what follows, specific half graphs provide counterexamples to
Conjecture~\ref{conj}. Let
$$(a_1,\ldots,a_6):=(1,0,-1,-1,0,1).$$

In what follows, for convenience, we will instead of column vectors
use row vectors, especially for eigenvectors.

Let
$$\x:=(x_1,\ldots,x_h)$$
where $x_i=a_s~\hbox{if}~i\equiv s\hspace{-0.25cm}\pmod6.$  In the
next theorem, we show that the vector $(\x, \x)$ (each $\x$
corresponds to a color class) is an eigenvector of a non-zero
eigenvalue
 of $H(h)$ for some $h$. In view of Remark~\ref{downer}, this disproves Conjecture~\ref{conj} .

\begin{theorem}\label{1} In any half graph $H(h)$, the vector $(\x, \x)$ is an eigenvector for $\la=1$ if $h\equiv1\pmod6$
and it is an eigenvector for $\la=-1$ if $h\equiv4\pmod6$.
\end{theorem}
\begin{proof}{
From Table~\ref{a_i}, we observe that for $1\le s\le6$,
$$\sum_{i=1}^{5-s}a_i=-a_s~~\hbox{and}~~\sum_{i=1}^{2-s}a_i=a_s,$$
where we consider $5-s$ and $2-s$ modulo 6 as elements of $\{1,\ldots,6\}.$
\begin{table}[ht]
$$\begin{array}{cccccc}
  \hline
   s&a_s & 5-s& \sum_{i=1}^{5-s}a_i  & 2-s&  \sum_{i=1}^{2-s}a_i  \\
   \hline
  1 & 1 & 4 & -1 & 1 & 1 \\
   2& 0 & 3 & 0 & 6 & 0 \\
  3 &-1 & 2 & 1 & 5 & -1 \\
  4 &-1 & 1 & 1 & 4 & -1 \\
  5 &0 & 6 & 0 & 3 & 0 \\
  6 & 1 & 5 & -1 & 2 & 1 \\
  \hline
\end{array}$$ \caption{The values of $\sum_{i=1}^{5-s}a_i$ and  $\sum_{i=1}^{2-s}a_i$}\label{a_i}
\end{table}

Note that, since $\sum_{i=1}^6a_i=0$, if $1\le\ell\le h$, $1\le s\le6$ and $\ell\equiv s\pmod6$, then
$$\sum_{i=1}^\ell x_i=\sum_{i=1}^s a_i.$$
Let $\{u_1,\ldots,u_h\}$ and $\{v_1,\ldots,v_h\}$ be the color
classes of $H(h)$. Let $h=6t+4$. We show that $(\x, \x)$ satisfies
the sum rule for $\la=-1$. By the symmetry, we only need to show
this for $u_i$'s. Let $i=6t'+s$ for some $1\le s\le6$. Then
$n-i+1=6(t-t')+5-s$.
$$  \sum_{j:\,v_j\sim u_i}x_j=\sum_{j=1}^{n-i+1}x_j=\sum_{j=1}^{5-s}a_j=-a_s=-x_i.$$

Now, let $h=6t+1$. We show that in this case $(\x,~\x)$ satisfies
the sum rule for $\la=1$. Let $i=6t'+s$ for some $1\le s\le6$. Then
$n-i+1=6(t-t')+2-s$.
$$  \sum_{j:\,v_j\sim u_i}x_j=\sum_{j=1}^{n-i+1}x_j=\sum_{j=1}^{2-s}a_j=a_s=x_i.$$
}\end{proof}

Now we give another class of counterexamples to Conjecture~\ref{conj}. For this, let $$\om^2+\om-1=0,$$ and
$$(b_1,\ldots,b_{10}):=(\om,-1,0,1,-\om,-\om,1,0,-1,\om).$$
Let $$\x:=(x_1,\ldots,x_h)$$ where $x_i=b_s~\hbox{if}~i\equiv
s\,\hspace{-0.25cm}\pmod{10}.$

\begin{theorem}\label{om} In any half graph $H(h)$, the vector $(\x, \x)$ is an eigenvector for $\la=\om$ if $h\equiv7\pmod{10}$ and it is an
 eigenvector for $\la=-\om$ if $h\equiv2\pmod{10}$.
\end{theorem}
\begin{proof}{From Table~\ref{b_i}, we observe that for $1\le s\le10$,
$$\sum_{i=1}^{8-s}b_i=\om b_s~~\hbox{and}~~\sum_{i=1}^{3-s}b_i=-\om b_s,$$
where we consider $8-s$ and $3-s$ modulo 10 as elements of $\{1,\ldots,10\}$.
\begin{table}[ht]
$${\small\begin{array}{cccccc}
\hline
 s &b_s &8-s &\sum_{i=1}^{8-s}b_i & 3-s&\sum_{i=1}^{3-s}b_i \\ \hline
1 & \om&7 &1-\om &2 & \om-1\\
 2 & -1& 6&-\om &1  &\om\\
 3 &0 & 5& 0 &10 &0\\
 4 &1 &4  &\om &9 &-\om\\
 5 & -\om& 3&\om-1 & 8& 1-\om\\
 6  &-\om &2 & \om-1&7 & 1-\om\\
 7  &1  &1 &\om &6 &-\om\\
 8  &0 &10 &0 &5 &0\\
 9  &-1 & 9 &-\om &4 &\om\\
  10 &\om &8 & 1-\om&3  &\om-1 \\ \hline
\end{array}}$$\caption{The values of $\sum_{i=1}^{8-s}b_i$ and  $\sum_{i=1}^{3-s}b_i $}\label{b_i}
\end{table}

Note that, since $\sum_{i=1}^{10}b_i=0$, if $1\le\ell\le k$, $1\le s\le10$ and $\ell\equiv s\pmod{10}$, then
$$\sum_{i=1}^\ell x_i=\sum_{i=1}^s b_i.$$
Let $k=10t+7$. Then $(\x, \x)$ satisfies the sum rule for $\la=\om$.
Let $i=10t'+s$ for some $1\le s\le10$. Then $n-i+1=10(t-t')+8-s$.
$$  \sum_{j:\,v_j\sim u_i}x_j=\sum_{j=0}^{n-i+1}x_j=\sum_{j=1}^{8-s}b_j=\om b_s=\om x_i.$$
Now, let $h=10t+2$. Assume that $i=10t'+s$ for some $1\le s\le10$. Then $n-i+1=6(t-t')+3-s$.
$$  \sum_{j:\,v_j\sim u_i}x_j=\sum_{j=1}^{n-i+1}x_j=\sum_{j=1}^{3-s}b_j=-\om b_s=-\om x_i.$$
It follows that in this case $(\x, \x)$ satisfies the sum rule for
$\la=-\om$. }\end{proof}
\medskip\noindent

\begin{remark}
The following two facts deserve to be mentioned:

\medskip

\noindent(i) Given $(\x,\x)$ as eigenvector of $H(h)$ for
$\la\in\{\pm1,\pm\om\}$, then $(\x,-\x)$ is an eigenvector of $H(h)$
for $-\la$. This gives more eigenvalues of $H(h)$ with eigenvectors
containing zero components.

\medskip\noindent
(ii) Let $\x$ be an eigenvector for eigenvalue $\la\neq 0$ of a
graph $G$ with $\x_v=0$, for some vertex $v$. If we add a new vertex
$u$ with $N(u)=N(v)$ and add a zero component to $\x$ corresponding
to $u$, then the new vector is an eigenvector of $H$ for $\la$. So,
we can extend any graph presented in Theorems~\ref{1} or \ref{om} to
construct infinitely many more counterexamples for
Conjecture~\ref{conj}.
\end{remark}

\section*{Acknowledgments} The research of the second author was in part supported by a grant
from IPM.

\end{document}